\documentclass[11pt]{amsart}
\usepackage{amssymb}

\theoremstyle{plain}
\newtheorem{theorem}{Theorem}
\newtheorem{corollary}{Corollary}

\newtheorem{lemma}{Lemma}

\theoremstyle{definition}

\theoremstyle{remark}

\numberwithin{equation}{section}

\def\K#1#2{\displaystyle{\mathop K\limits_{#1}^{#2}}}

\setbox0=\hbox{$+$}

\newdimen\plusheight
\plusheight=\ht0
\def\+{\;\lower\plusheight\hbox{$+$}\;}

\setbox0=\hbox{$-$}
\newdimen\minusheight
\minusheight=\ht0
\def\-{\;\lower\minusheight\hbox{$-$}\;}

\setbox0=\hbox{$\cdots$}
\newdimen\cdotsheight
\cdotsheight=\plusheight
\def\cds{\lower\cdotsheight\hbox{$\cdots$}}

\begin{document}

%\markboth{Douglas Bowman, James Mc Laughlin, Nancy J. Wyshinski}
%{A $q$-Continued Fraction}

%\catchline{}{}{}{}{}
%
%%%%%%%%%%%%%%%%%%%%%%%%%%%%%%%%%%%%%%%%%%%%%%%%%%%%%%%%%%%%%%%%%%%%

%\title{INSTRUCTIONS FOR TYPESETTING
%MANUSCRIPTS\footnote{For the title, try not to
%use more than 3 lines. Typeset in 10 pt roman, uppercase and boldface.}
%}

\title{A $Q$-CONTINUED FRACTION}

 \author{DOUGLAS BOWMAN}
\address{ Northern Illinois University\\
   Mathematical Sciences\\
   DeKalb, IL 60115-2888 }
 \email{bowman@math.niu.edu}

 \author{JAMES MC LAUGHLIN}
\address{Mathematics Department\\
 25 University Avenue\\
West Chester University, West Chester, PA 19383}
\email{jmclaughl@wcupa.edu}

\author{ NANCY J. WYSHINSKI}
\address{Mathematics Department\\
       Trinity College\\
        300 Summit Street, Hartford, CT 06106-3100}
\email{nancy.wyshinski@trincoll.edu}

\begin{abstract}
We use the method of generating functions to find the limit of a
$q$-continued fraction, with 4 parameters, as a ratio of certain
$q$-series.

We then use this result to give new proofs of several known
continued fraction identities, including  Ramanujan's continued
fraction expansions for $(q^2;q^3)_{\infty}/(q;q^3)_{\infty}$ and
$(q;q^{2})_{\infty}/(q^{3};q^{6})_{\infty}^{3}$. In addition, we
give a new proof of the famous Rogers-Ramanujan identities.

We also use our main result to derive two generalizations of
another continued fraction due to Ramanujan.
\end{abstract}



\thanks{The research of the first author was partially supported by National Science Foundation grant
DMS-0300126.}

\maketitle

\section{Introduction}
The work in the present paper was initially motivated by claims of
Ramanujan about two related continued fractions.

The first claim concerns a curious continued fraction with three
limits. To describe Ramanujan's claim, found his lost notebook
(\cite{S88}, p. 45), we first need some notation.
 Throughout take $q\in\mathbb{C}$
with $|q|<1$. The following standard notation for $q$-products
will also be employed:
\begin{align*}
&(a)_{0}:=(a;q)_{0}:=1,& &
(a)_{n}:=(a;q)_{n}:=\prod_{k=0}^{n-1}(1-a\,q^{k}),& & \text{ if }
n \geq 1,&
\end{align*}
and
\begin{align*}
&(a;q)_{\infty}:=\prod_{k=0}^{\infty}(1-a\,q^{k}),& & |q|<1.&
\end{align*}
Set $\omega = e^{2 \pi i/3}$. Ramanujan's claim was that, for
$|q|<1$, {\allowdisplaybreaks \small{
\begin{equation}\label{3lim1}
\lim_{n \to \infty}
 \left (
\frac{1}{1}
 \-
\frac{1}{1+q} \- \frac{1}{1+q^2} \- \cds \- \frac{1}{1+q^n +a}
\right ) = -\omega^{2} \left ( \frac{\Omega - \omega^{n+1}}{\Omega
- \omega^{n-1}} \right ).
\frac{(q^{2};q^{3})_{\infty}}{(q;q^{3})_{\infty}},
\end{equation}
} } where {\allowdisplaybreaks
\begin{equation*}
\Omega :=\frac{1-a\omega^{2}}{1-a \omega}
\frac{(\omega^{2}q;q)_{\infty}}{(\omega q;q)_{\infty}}.
\end{equation*}
} Ramanujan's notation is confusing, but what his claim means is
that the limit exists as $n \to \infty$ in each of the three
congruence classes modulo 3, and that the limit is given by the
expression on the right side of (\ref{3lim1}). Ramanujan's claim
is proved in a recent paper \cite{ABSYZ02}.

In \cite{BML05} the first and second authors of the present paper
generalized Ramanujan's result to produce $q$-continued fractions
with $m$ limits, where $m \geq 3$ is an arbitrary integer. Let
$\omega$ be a primitive $m$-th root of unity and, for ease of
notation,
  let $\bar{ \omega} = 1/\omega$.  Define
{\allowdisplaybreaks
\begin{equation}\label{rammlim}
G(q):= \frac{1}{1} \- \frac{1}{\omega + \bar{ \omega}+q} \-
\frac{1}{\omega + \bar{ \omega}+q^2} \- \frac{1}{\omega + \bar{
\omega}+q^3} \+ \cds.
\end{equation}
} For $|q|<1$ and $a,\,x \not = 0$, define
\begin{equation}\label{Peq}
P(a,x,q):= \sum_{j=0}^{\infty}\frac{q^{j(j+1)/2}a^{j}x^{j}}
{(q;q)_{j}(x^2q;q)_{j}}.
\end{equation}
In \cite{BML05}  the following theorem is proved.
\begin{theorem}\label{ramgentheor}
Let $\omega$ be a primitive $m$-th root of unity and
  let $\bar{ \omega} = 1/\omega$.
Let $1\leq i \leq m$. Then {\allowdisplaybreaks
\begin{multline}\label{ramgentheoreq}
\lim_{k \to \infty} \frac{1}{\omega + \bar{ \omega}+q} \-
\frac{1}{\omega + \bar{ \omega}+q^2} \- \cds \-
\frac{1}{\omega + \bar{ \omega}+q^{mk+i}}\\
= \frac{\omega^{1-i}P(q,\omega,q)-\omega^{i-1}P(q,w^{-1},q)}
{\omega^{-i}P(1,\omega,q)-\omega^{i}P(1,w^{-1},q)}.
\end{multline}
}
\end{theorem}
The result is stated for the first tail of $G(q)$, rather than
$G(q)$ itself, for aesthetic reasons. Ramanujan's continued
fraction is the special case $m=6$ of \eqref{rammlim}. The result
at \eqref{ramgentheoreq} also appears in the paper of Ismail and
Stanton \cite{IS05}.

If we make the transformation $q\to 1/q$ in Ramanujan's continued
fraction \begin{equation}\label{3limcf}
 T(q):=\frac{1}{1}
 \-
\frac{1}{1+q} \- \frac{1}{1+q^2} \- \cds \- \frac{1}{1+q^n} \-
\cds
\end{equation}
and clear denominators, we get the continued fraction
 \[
S(q):= \frac{1}{1}
 \-
\frac{q}{1+q} \- \frac{q^{3}}{1+q^2} \- \frac{q^{5}}{1+q^3} \-
\cds \- \frac{q^{2n-1}}{1+q^n} \- \cds .
\]
Remarkably, Ramanujan made a deep claim about $S(q)$ also, namely,
that
\begin{equation}\label{ramq2q3}
S(q)=\frac{(q^{2};q^{3})_{\infty}}{(q;q^{3})_{\infty}}.
\end{equation}
This claim is proved in \cite{ABSYZ03}, by the same group of
authors who proved \eqref{3lim1} in \cite{ABSYZ02}. In this paper
the authors remark that ``Of the many continued fractions found by
Ramanujan, \eqref{ramq2q3} is, by far, the most difficult to
prove." The only other proof, to date, can be found \cite{ABJL92}.
Amongst other things, we give a new proof of \eqref{ramq2q3} in
the present paper.

While the first and second authors of this current paper were
working on generalizations of \eqref{3limcf}, computer
investigations indicated to the second and third author that the
numerator and denominator of $S(q)$ converged separately to
$1/(q;q^{3})_{\infty}$ and $1/(q^{2};q^{3})_{\infty}$,
respectively. That this is indeed the case is not obvious in the
proofs in \cite{ABJL92} or \cite{ABSYZ03}. While trying to prove
these separate convergence results, at some point ``the penny
dropped" and we realized that we could adapt the method being used
to generalize \eqref{3limcf} by making the substitution $q \to
1/q$ at the right point and obtain the limit (as quotients of
certain $q$-series) of a quite general class of $q$-continued
fractions, a class which includes $S(q)$ above. Ramanujan's result
\eqref{ramq2q3} then followed after applying a special case of a
powerful basic hypergeometric series transformation due to Watson.

On the way to our main result, we find the limit of  continued
fractions of the form
\[ H(a,b,c,d,q)= \frac{1}{1} \+ \frac{-ab+cq}{a+b+dq}
\+\frac{-ab+cq^2}{a+b+dq^2} \+ \cds \+ \frac{-ab+cq^n}{a+b+dq^n}\+
\cds .
\]
Several of the $q$-continued fractions in the literature arise as
special cases of this continued fraction. However, we do not
investigate this here as a closely related continued fraction,
namely,
\[  1+a +d+ \frac{-a+cq}{a+1+dq}
\+\frac{-a+cq^2}{a+1+dq^2} \+ \cds \+ \frac{-a+cq^n}{a+1+dq^n}\+
\cds,
\]
was investigated by Hirschhorn in \cite{H74} and \cite{HT80}, and
many of the well-known continued fraction identities were derived
by him as corollaries of his main result.

Our result for the continued fraction $H(a,b,c,d,q)$ could be
derived from Hirschhorn's through various substitutions,
equivalence transformations and series manipulations. However, it
is perhaps just as simple to derive it directly, using the
recurrence relations for the numerators and denominators and
generating functions. More importantly, the derivation of our main
result relies on finding expressions for the $N$-th numerator and
denominator of $H(a,b,c,d,q)$ and then applying the transformation
$q \to 1/q$.

Our main result is as follows (see Theorem \ref{1/qth} for a more
complete statement): Let $a$, $b$, $c$, $d$ be complex numbers
with $d \not = 0$ and $|q|<1$. Define \[ H_{1}(a,b,c,d,q):=
\frac{1}{1} \+ \frac{-abq+c}{(a+b)q+d}
 \+ \cds \+ \frac{-ab
q^{2n+1}+cq^n}{(a+b) q^{n+1}+d}\+ \cds .
\]
Then $H_{1}(a,b,c,d,q)$ converges and
\begin{equation*}
\frac{1}{H_{1}(a,b,c,d,q)}-1= \frac{c-abq}{d+aq} \frac{\sum_{j
=0}^{\infty}\displaystyle{\frac{(b/d)^{j}(-c/bd)_{j}\,q^{j(j+3)/2}}{(q)_{j}(-aq^2/d)_{j}}
}} {\sum_{j
=0}^{\infty}\displaystyle{\frac{(b/d)^{j}(-c/bd)_{j}\,q^{j(j+1)/2}}{(q)_{j}(-aq/d)_{j}}}
}.
\end{equation*}
We then use this result, in combination with other well know
transformations like the Jacobi triple product identity and the
aforementioned transformation of Watson, to deduce various
corollaries, some of which are originally due to Ramanujan.
\begin{equation}
\frac{1}{1}
 \-
\frac{q}{1+q} \- \frac{q^{3}}{1+q^2} \- \frac{q^{5}}{1+q^3} \-
\cds \- \frac{q^{2n-1}}{1+q^n} \- \cds
=\frac{(q^{2};q^{3})_{\infty}}{(q;q^{3})_{\infty}}.
\end{equation}
\begin{multline}
1+ \frac{aq}{1} \+ \frac{bq+e}{1} \+ \frac{aq^2}{1} \+
\frac{bq^2+e}{1} \+ \frac{aq^3}{1} \+ \frac{bq^3+e}{1} \+ \cds \\
= \frac{\displaystyle{ \sum_{j
=0}^{\infty}\frac{\left(a/(e+1)\right)^{j}\left(eq/(e+1)\right)_{j}\,q^{j(j+1)/2}}{(q)_{j}\left
(-bq/(e+1)\right )_{j}}}}{\displaystyle{\sum_{j
=0}^{\infty}\frac{\left(aq/(e+1)\right)^{j}\left(e/(e+1)\right)_{j}\,q^{j(j+1)/2}}{(q)_{j}\left
(-bq/(e+1) \right )_{j}} }}.
\end{multline}
\begin{multline}
1+ \frac{aq+e}{1} \+ \frac{bq}{1} \+ \frac{aq^2+e}{1} \+
\frac{bq^2}{1} \+ \frac{aq^3+e}{1} \+ \frac{bq^3}{1} \+ \cds \\
= \frac{\displaystyle{(e+1) \sum_{j
=0}^{\infty}\frac{\left(a/(e+1)\right)^{j}\left(eb/(a(e+1))\right)_{j}\,q^{j(j+1)/2}}{(q)_{j}\left
(-bq/(e+1)\right )_{j}}}}{\displaystyle{\sum_{j
=0}^{\infty}\frac{\left(aq/(e+1)\right)^{j}\left(eb/(a(e+1))\right)_{j}\,q^{j(j+1)/2}}{(q)_{j}\left
(-bq/(e+1) \right )_{j}} }}.
\end{multline}
Remark: each of the previous two continued fractions generalizes a
continued fraction of Ramanujan (see \eqref{phicf}).
\begin{equation}
(-aq)_{\infty} \sum_{j
=0}^{\infty}\frac{(bq)^{j}(-c/b)_{j}\,q^{j(j-1)/2}}{(q)_{j}(-aq)_{j}}
=(-bq)_{\infty} \sum_{j
=0}^{\infty}\frac{(aq)^{j}(-c/a)_{j}\,q^{j(j-1)/2}}{(q)_{j}(-bq)_{j}}.
\end{equation}
\begin{align}
 \frac{1}{1} \+
 \frac{q+q^{2}}{1}
\+
 \frac{q^{2}+q^{4}}{1}
\+
 \frac{q^{3}+q^{6}}{1}
%\+
 %\frac{q^{4}+q^{8}}{1}
\+ \cds =\frac{(q;q^{2})_{\infty}}{(q^{3};q^{6})_{\infty}^{3}}.
\end{align}
We also give a proof of the Rogers-Ramanujan identities.
\begin{align}
\sum_{n=0}^{\infty}\frac{q^{n^{2}}}{(q;q)_{n}}&=\frac{1}{(q;q^{5})_{\infty}(q^{4}
;q^{5})_{\infty}}
,\\
&\phantom{as} \notag \\
\sum_{n=0}^{\infty}\frac{q^{n^{2}+n}}{(q;q)_{n}}
&=\frac{1}{(q^{2};q^{5})_{\infty}(q^{3};q^{5})_{\infty}}. \notag
\end{align}
Remarks: (1) We could, without loss of generality in
$H_{1}(a,b,c,d,q)$, set one of the parameters $a$, $b$, $c$ or $d$
equal to one and recover the general case, if desired,   by a
change of variables. However, it is better for our present
purposes to retain the flexibility of having 4 parameters and not
having deal with these  changes of variables (see Corollaries
\ref{cram} and \ref{cphi}, for example, where having this full
flexibility allowed us to derive our results more easily).

(2) Shortly after we had proved our main result for
$H_{1}(a,b,c,d,q)$ and obtained the various corollaries, the
second author was browsing an early draft version of \cite{AB05},
which one of the authors of \cite{AB05} had given him. The purpose
was to find further applications of our main result about
$H_{1}(a,b,c,d,q)$ and possibly give new proofs, or possibly
generalizations (see Corollary \ref{cphi}), of some of Ramanujan's
results. Instead he was surprised to find that Ramanujan had
stated a result that was quite close to our main result for
$H_{1}(a,b,c,d,q)$:

For any complex numbers $a$, $b$, $\lambda$ and $q$, with $|q|<1$,
define
\[
G(a,\lambda; b;q)=\sum_{n=0}^{\infty} \frac{(-\lambda
/a;q)_{n}a^{n}q^{n(n+1)/2}} {(q;q)_{n}(-bq;q)_{n}}.
\]
\textbf{Entry 6.4.4 (p.43)} We have
\begin{multline}\label{ramsim}
\frac{G(a q, \lambda q; b; q)}{G(a , \lambda ; b; q)} =
\frac{1}{1+aq} \+ \frac{\lambda q - a b q^{2}}{1+q(a q +b)} \+
\frac{\lambda q^2 - a b q^{4}}{1+q^2(a q +b)}\\ \+ \cds \+
\frac{\lambda q^n - a b q^{2n}}{1+q^n(a q +b)}\+ \cds .
\end{multline}
It seems clear that our main result concerning $H_{1}(a,b,c,d,q)$
could also be derived from \eqref{ramsim}, after various changes
of variable and $q$-series manipulations, but possibly proving it
the way we did may be more illuminating.

\section{Proofs}

We recall the $q$-binomial theorem (\cite{A76}, pp. 35--36).
\begin{lemma}\label{qbin}
If $\left [
\begin{matrix}
n\\
m
\end{matrix}
\right ] $ denotes the Gaussian polynomial defined by
\[
\left [
\begin{matrix}
n\\
m
\end{matrix}
\right ] := \left [
\begin{matrix}
n\\
m
\end{matrix}
\right ]_{q} :=
\begin{cases}
\displaystyle{
\frac{(q;q)_{n}}{(q;q)_{m}(q;q)_{n-m}}}, &\text{ if } 0 \leq m \leq n,\\
0, &\text{ otherwise },
\end{cases}
\]
then
\begin{align}\label{qbineq}
&(z;q)_{N}= \sum_{j=0}^{N}\left [
\begin{matrix}
N\\
j
\end{matrix}
\right ]
(-1)^{j}z^{j}q^{j(j-1)/2},\\
&\frac{1}{(z;q)_{N}} =\sum_{j=0}^{\infty} \left [
\begin{matrix}
N+j-1\\
j
\end{matrix}
\right ] z^{j}. \notag
\end{align}
\end{lemma}
Note for later use that
\[
\left [
\begin{matrix}
n\\
m
\end{matrix}
\right ]_{1/q} =q^{m(m-n)}\left [
\begin{matrix}
n\\
m
\end{matrix}
\right ]_{q}.
\]
\begin{theorem}\label{qth}
Let \[ H(a,b,c,d,q)= \frac{1}{1} \+ \frac{-ab+cq}{a+b+dq}
\+\frac{-ab+cq^2}{a+b+dq^2} \+ \cds \+ \frac{-ab+cq^n}{a+b+dq^n}\+
\cds .
\]
(i) Let $A_{N}:=A_{N}(q)$ and $B_{N}:=B_{N}(q)$ denote the $N$-th
numerator convergent and $N$-th denominator convergent,
respectively, of $H(a,b,c,d,q)$. Then $A_{N}$ and $B_{N}$ are
given explicitly by the formulae
\begin{multline}\label{ANeq}
A_{N}= \\\sum_{j,l,n\geq 0}a^{j}b^{N-1-n-j-l}c^{l}
d^{n-l}q^{n(n+1)/2+l(l+1)/2} \left [
\begin{matrix}
n+j\\
j
\end{matrix}
\right ] \left [
\begin{matrix}
N-1-j-l\\
n
\end{matrix}
\right ] \left [
\begin{matrix}
n\\
l
\end{matrix}
\right ].
\end{multline}
For $N\geq 2$,
\begin{multline}\label{BNeq}
 B_{N}= A_{N}+ (c q-a b) \times \\
\sum_{j,l,n\geq 0}a^{j}b^{N-2-n-j-l}c^{l}
d^{n-l}q^{n(n+3)/2+l(l+1)/2} \left [
\begin{matrix}
n+j\\
j
\end{matrix}
\right ] \left [
\begin{matrix}
N-2-j-l\\
n
\end{matrix}
\right ] \left [
\begin{matrix}
n\\
l
\end{matrix}
\right ].
\end{multline}
(ii) If $|a/b| <1$ and $|q|<1$ then $H(a,b,c,d,q)$ converges and
\begin{equation}\label{Hlim}
\frac{1}{H(a,b,c,d,q)}-1= \frac{\displaystyle{ (c q/b
-a)\sum_{n=0}^{\infty} \frac{(d/b)^{n}q^{n(n+3)/2}(-cq/db)_{n}}
{(a/b)_{n+1}(q)_{n}}}}{\displaystyle{\sum_{n=0}^{\infty}
\frac{(d/b)^{n}q^{n(n+1)/2}(-cq/d b)_{n}} {(a/b)_{n+1}(q)_{n}} }}.
\end{equation}
(iii) If $|q|<1$, $|a|<1$ and $b=1$, then the numerators and
denominators converge separately and
\begin{equation}\label{ANlim}
\lim_{N \to \infty} A_{N} = \sum_{n=0}^{\infty}
\frac{d^{n}q^{n(n+1)/2}(-cq/d)_{n}} {(a)_{n+1}(q)_{n}},
\end{equation}
\begin{equation}\label{BNlim}
\lim_{N \to \infty} B_{N} = \sum_{n=0}^{\infty}
\frac{d^{n}q^{n(n+1)/2}(-cq/d)_{n}} {(a)_{n+1}(q)_{n}} + (c q
-a)\sum_{n=0}^{\infty} \frac{d^{n}q^{n(n+3)/2}(-cq/d)_{n}}
{(a)_{n+1}(q)_{n}}.
\end{equation}
\end{theorem}
Remarks: (a) By symmetry the conditions on $a$ and $b$ in (ii) and
(iii)  can be interchanged, in which case $a$ and $b$ are
interchanged on the right
sides.\\
(b) The left side in (ii) is displayed as shown to simplify the
representation on the right side.
\begin{proof}
(i) We follow the method of Hirschhorn in \cite{H74}. We suppose
initially that $|q|<1$ and for ease of notation, let
$A_{N}:=A_{N}(q)$ and $B_{N}:=B_{N}(q)$ and set $F(t) = \sum_{N
\geq 1}A_{N}t^{n}$ and $G(t) = \sum_{N \geq 1}B_{N}t^{n}$. We
suppose initially that $a$ and $b$ are chosen so that the series
defining $F(t)$ and $G(t)$ converge ($|a|=|b|=1$ suffices for
this). From the recurrence relations for the convergents of a
continued fraction, we have that
\begin{equation*}
A_{N+1}=(a+b+d q^{N})A_{N}+(-a b + c q^{N})A_{N-1}
\end{equation*}
holds for $N\geq 1$. If this equation is multiplied by $t^{N+1}$
and summed over $N\geq 1$, we get
\begin{equation*}
F(t)-t=t(a+b)F(t)+t d F(t q)-a b t^2 F(t) + c t^2 q F(t q).
\end{equation*}
Here we have used $A_{1}=1$ and $A_{0}=0$. The last equation can
be rewritten to give
\begin{equation*} F(t) = \frac{t}{(1-a t)(1 -
b t)} +\frac{t(d + c t q)}{(1-a t)(1 - b t)}F(t q).
\end{equation*}
Next, this equation is iterated and we use the fact that $F(0) =0$
to get that
\begin{equation*}
F(t) = \sum_{n\geq
1}\frac{t^{n}d^{n-1}(-ctq/d)_{n-1}q^{n(n-1)/2}}{(at)_{n}(bt)_{n}}.
\end{equation*}
The $q$-binomial theorem (Lemma \ref{qbin}) is applied to the
$q$-products in the expression above to give that
\begin{multline*}
F(t) =\\
 \sum_{\stackrel{n\geq 1,}{j,k,l\geq 0}}
t^{n+j+k+l}d^{n-1-l}a^{j}b^{k}c^{l}q^{\frac{n(n-1)}{2}+\frac{l(l+1)}{2}}
\left [
\begin{matrix}
n+j-1\\
j
\end{matrix}
\right ] \left [
\begin{matrix}
n+k-1\\
k
\end{matrix}
\right ] \left [
\begin{matrix}
n-1\\
l
\end{matrix}
\right ]\\
=\sum_{j,k,l,n\geq 0}
t^{n+1+j+k+l}d^{n-l}a^{j}b^{k}c^{l}q^{\frac{n(n+1)}{2}+\frac{l(l+1)}{2}}
\left [
\begin{matrix}
n+j\\
j
\end{matrix}
\right ] \left [
\begin{matrix}
n+k\\
n
\end{matrix}
\right ] \left [
\begin{matrix}
n\\
l
\end{matrix}
\right ].
\end{multline*}
Finally, we let $N=n+j+k+l+1$,  substitute for $k$
 and use the definition of $F(t)$ to get
\eqref{ANeq}.

By similar reasoning we get that {\allowdisplaybreaks
\begin{align*}
G(t) &= \sum_{n\geq
1}\frac{t^{n}d^{n-1}(-ctq/d)_{n-1}q^{n(n-1)/2}}{(at)_{n}(bt)_{n}}(1+(cq-ab)tq^{n-1})\\
&=F(t)+(cq-ab)\sum_{n\geq
1}\frac{t^{n+1}d^{n-1}(-ctq/d)_{n-1}q^{(n-1)(n+2)/2}}{(at)_{n}(bt)_{n}}\\
&=F(t)+(cq-ab)\sum_{n\geq
0}\frac{t^{n+2}d^{n}(-ctq/d)_{n}q^{n(n+3)/2}}{(at)_{n+1}(bt)_{n+1}}\\
&=F(t)+(cq-ab)\sum_{j,k,l,n\geq 0}
t^{n+2+j+k+l}d^{n-l}a^{j}b^{k}c^{l}q^{\frac{n(n+3)}{2}+\frac{l(l+1)}{2}}\\
&\phantom{ASASASASASfffsdfsdfsdASASAS} \times \left [
\begin{matrix}
n+j\\
j
\end{matrix}
\right ] \left [
\begin{matrix}
n+k\\
n
\end{matrix}
\right ] \left [
\begin{matrix}
n\\
l
\end{matrix}
\right ].
\end{align*}
} We let $N=n+j+k+l+2$,  substitute for $k$
 and use the definition of $G(t)$ to get
\eqref{BNeq}.

%\newpage

(ii) The expression for $A_{N}$ in \eqref{ANeq} can be re-written
as {\allowdisplaybreaks
\begin{align}\label{ANnew}
A_{N}&=b^{N-1}\sum_{n\geq 0} (d/b)^{n}q^{n(n+1)/2}
 \sum_{j\geq 0}
\left [
\begin{matrix}
n+j\\
j
\end{matrix}
\right ](a/b)^{j}\\
&\phantom{sadasdasdadadssdfdsfgdfga}\times
 \sum_{l\geq 0}
 \left [
\begin{matrix}
N-1-j-l\\
n
\end{matrix}
\right ]
 q^{l(l-1)/2}
 \left (
\frac{cq}{bd}
 \right )^{l}
\left [
\begin{matrix}
n\\
l
\end{matrix}
\right ]. \notag
\end{align}
}

By the definition of the Gaussian polynomials in Lemma \ref{qbin},
$j$, $l$ and $n$ are restricted by $l\leq n$ and $l+j+n\leq N-1$.

Similarly, the expression for $B_{N}$ in \eqref{BNeq} can be
re-written as
\begin{align*}
B_{N}&=A_{N}+b^{N-1}(cq/b-a)\sum_{n\geq 0} (d/b)^{n}q^{n(n+3)/2}
 \sum_{j\geq 0}
\left [
\begin{matrix}
n+j\\
j
\end{matrix}
\right ](a/b)^{j}\\
&\phantom{sadasdasdadadssdfdsfgdfga}\times
 \sum_{l\geq 0}
 \left [
\begin{matrix}
N-2-j-l\\
n
\end{matrix}
\right ]
 q^{l(l-1)/2}
 \left (
\frac{cq}{bd}
 \right )^{l}
\left [
\begin{matrix}
n\\
l
\end{matrix}
\right ].
\end{align*}
Thus
 {\allowdisplaybreaks
\begin{align*}
\lim_{N\to \infty}
\frac{B_{N}-A_{N}}{b^{N-1}}&=(cq/b-a)\sum_{n\geq 0}
(d/b)^{n}q^{n(n+3)/2}
 \sum_{j\geq 0}
\left [
\begin{matrix}
n+j\\
j
\end{matrix}
\right ](a/b)^{j}\\
&\phantom{sadasdasdadadssdfdsfgdfga}\times
 \sum_{l =0}^{n}\frac{
  q^{l(l-1)/2}
 \left (
\frac{cq}{bd}
 \right )^{l}}{(q)_{n}}
\left [
\begin{matrix}
n\\
l
\end{matrix}
\right ]\\
&= (c q/b -a)\sum_{n=0}^{\infty}
\frac{(d/b)^{n}q^{n(n+3)/2}(-cq/db)_{n}} {(a/b)_{n+1}(q)_{n}}.
\end{align*}
} Likewise the expression for $A_{N}$ above gives that
\begin{equation*}
\lim_{N\to \infty} \frac{A_{N}}{b^{N-1}} = \sum_{n=0}^{\infty}
\frac{(d/b)^{n}q^{n(n+1)/2}(-cq/d b)_{n}} {(a/b)_{n+1}(q)_{n}}.
\end{equation*}
Equation \ref{Hlim} is now immediate.

(iii) If we now set $b=1$ the limit above, we get immediately that
\begin{equation*}
\lim_{N \to \infty} A_{N} =\sum_{n\geq 0}\frac{
d^{n}q^{n(n+1)/2}(-cq/d)_{n}} {(a)_{n+1}(q)_{n}}.
\end{equation*}
This proves \eqref{ANlim}. The proof of \eqref{BNlim} is similar.
\end{proof}

\begin{theorem}\label{1/qth}
Let $a$, $b$, $c$, $d$ be complex numbers with $d \not = 0$ and
$|q|<1$. Define \[ H_{1}(a,b,c,d,q):= \frac{1}{1} \+
\frac{-abq+c}{(a+b)q+d}
 \+ \cds \+ \frac{-ab
q^{2n+1}+cq^n}{(a+b) q^{n+1}+d}\+ \cds .
\]
(i) Let $C_{N}:=C_{N}(q)$ and $D_{N}:=D_{N}(q)$ denote the $N$-th
numerator convergent and $N$-th denominator convergent,
respectively, of $H_{1}(a,b,c,d,q)$. Then $C_{N}$ and $D_{N}$ are
given explicitly by the following formulae.
\begin{multline}\label{CNeq}
C_{N}= d^{N-1}\sum_{j,\,l ,\,n\geq 0}a^{j}b^{n-j-l}c^{l}
d^{-n-l}q^{n(n+1)/2+l(l-1)/2} \\
 \times\left [
\begin{matrix}
N-1-n+j\\
j
\end{matrix}
\right ]_{q} \left [
\begin{matrix}
N-1-j-l\\
n-j-l
\end{matrix}
\right ]_{q} \left [
\begin{matrix}
N-1-n\\
l
\end{matrix}
\right ]_{q}.
\end{multline}
For $N\geq 2$,
\begin{multline}\label{DNeq}
 D_{N}=C_{N}+(c/bq-a)\sum_{j,\,l,  \, n \geq 0}a^{j}b^{n+1-j-l}c^{l}
d^{N-2-n-l} \times \\
q^{(n+1)(n+2)/2+l(l-1)/2} \left [
\begin{matrix}
N-2-n+j\\
j
\end{matrix}
\right ]_{q} \left [
\begin{matrix}
N-2-j-l\\
n-j-l
\end{matrix}
\right ]_{q} \left [
\begin{matrix}
N-2-n\\
l
\end{matrix}
\right ]_{q}.
 \end{multline}
(ii) If $|q|<1$ then $H_{1}(a,b,c,d,q)$ converges and
\begin{equation}\label{H1lim}
\frac{1}{H_{1}(a,b,c,d,q)}-1= \frac{c-abq}{(d+aq)q} \frac{\sum_{j
=0}^{\infty}\displaystyle{\frac{(b/d)^{j}(-c/bd)_{j}\,q^{(j+1)(j+2)/2}}{(q)_{j}(-aq^2/d)_{j}}
}} {\sum_{j
=0}^{\infty}\displaystyle{\frac{(b/d)^{j}(-c/bd)_{j}\,q^{j(j+1)/2}}{(q)_{j}(-aq/d)_{j}}}
}.
\end{equation}
(iii) If $|q|<1$ and $d=1$, then the numerators and denominators
converge separately and
\begin{equation}\label{CNlim}
C_{\infty}:=\lim_{N \to \infty} C_{N} =(-aq)_{\infty}
\sum_{j=0}^{\infty}
 \frac{ q^{j(j+1)/2}b^{j}(-c/b)_{j}}{(q)_{j}(-aq)_{j}}.
\end{equation}
\begin{equation}\label{DNlim}
D_{\infty}:=\lim_{N \to \infty} D_{N} =
C_{\infty}+(c/q-ab)(-aq)_{\infty} \sum_{j=0}^{\infty}
 \frac{ q^{(j+1)(j+2)/2}b^{j}(-c/b)_{j}}{(q)_{j}(-aq)_{j+1}}.
\end{equation}
\end{theorem}
\begin{proof}
The continued fraction $H_{1}(a,b,c,d,q)$ is derived from
$H(a,b,c,d,q)$ by making the substitution $q \to 1/q$ and then
applying a sequence of similarity transformations to clear the
negative powers of $q$. Thus, {\allowdisplaybreaks
\begin{align*}
C_{N}&=q^{N(N-1)/2}A_{N}(1/q),\\
&=\sum_{j,l,n\geq 0}a^{j}b^{N-1-n-j-l}c^{l}
d^{n-l}q^{(N-n)(N-n-1)/2+l(l-1)/2} \\
&\phantom{dfasdfasdfdsfsadfas} \times\left [
\begin{matrix}
n+j\\
j
\end{matrix}
\right ]_{q} \left [
\begin{matrix}
N-1-j-l\\
n
\end{matrix}
\right ]_{q} \left [
\begin{matrix}
n\\
l
\end{matrix}
\right ]_{q}\\
&=\sum_{j,l,n\geq 0}a^{j}b^{n-j-l}c^{l}
d^{N-1-n-l}q^{n(n+1)/2+l(l-1)/2} \\
&\phantom{dfasdfasdfdsadfas} \times\left [
\begin{matrix}
N-1-n+j\\
j
\end{matrix}
\right ]_{q} \left [
\begin{matrix}
N-1-j-l\\
n-j-l
\end{matrix}
\right ]_{q} \left [
\begin{matrix}
N-1-n\\
l
\end{matrix}
\right ]_{q}\\
&=d^{N-1}\sum_{n=0}^{N-1}\sum_{j=0}^{n}\sum_{l=0}^{\min(n-j,\,N-1-n)}a^{j}b^{n-j-l}c^{l}
d^{-n-l}q^{n(n+1)/2+l(l-1)/2} \\
&\phantom{dfasdfasdfdsadfas} \times\left [
\begin{matrix}
N-1-n+j\\
j
\end{matrix}
\right ]_{q} \left [
\begin{matrix}
N-1-j-l\\
n-j-l
\end{matrix}
\right ]_{q} \left [
\begin{matrix}
N-1-n\\
l
\end{matrix}
\right ]_{q}.
\end{align*}
} For the next-to-last step we replaced $n$ by $N-1-n$ and in the
last step the upper limits on $j$, $l$ and $n$ come from the
definition of the Gaussian polynomials in Lemma \ref{qbin}. This
proves \eqref{CNeq}.

Similarly, {\allowdisplaybreaks
\begin{align*}
 D_{N}&=q^{N(N-1)/2}B_{N}(1/q)\\
&=C_{N} + (c/bq-a)\sum_{j,l,n\geq 0}a^{j}b^{N-1-n-j-l}c^{l}
d^{n-l}q^{(N-n)(N-n-1)/2+l(l-1)/2} \\
&\phantom{dfasdfasdfdsfsadasdasfas} \times\left [
\begin{matrix}
n+j\\
j
\end{matrix}
\right ]_{q} \left [
\begin{matrix}
N-2-j-l\\
n
\end{matrix}
\right ]_{q} \left [
\begin{matrix}
n\\
l
\end{matrix}
\right ]_{q}\\
&=C_{N}+(c/bq-a)\sum_{j,l,n\geq 0}a^{j}b^{n+1-j-l}c^{l}
d^{N-2-n-l}q^{(n+1)(n+2)/2+l(l-1)/2} \\
&\phantom{dfasdfasdfdsadfas} \times\left [
\begin{matrix}
N-2-n+j\\
j
\end{matrix}
\right ]_{q} \left [
\begin{matrix}
N-2-j-l\\
n-j-l
\end{matrix}
\right ]_{q} \left [
\begin{matrix}
N-2-n\\
l
\end{matrix}
\right ]_{q}\\
&=C_{N}+(c/bq-a)\sum_{n= 0}^{N-2} \sum_{j= 0}^{n} \sum_{l=
0}^{\min ( n-j,\ N-2-n)} a^{j}b^{n+1-j-l}c^{l}
d^{N-2-n-l} \times \\
&q^{(n+1)(n+2)/2+l(l-1)/2} \left [
\begin{matrix}
N-2-n+j\\
j
\end{matrix}
\right ]_{q} \left [
\begin{matrix}
N-2-j-l\\
n-j-l
\end{matrix}
\right ]_{q} \left [
\begin{matrix}
N-2-n\\
l
\end{matrix}
\right ]_{q}.
\end{align*}
} The second equality follows upon replacing $n$ by $N-n-2$ and
the bounds on $j$, $l$ and $n$ in the last equality follow, as
above, from the definition of the Gaussian polynomials in Lemma
\ref{qbin}. This proves \eqref{DNeq}.

From \eqref{DNeq}, {\allowdisplaybreaks
\begin{align*}
\lim_{N \to \infty}& \frac{D_{N}-C_{N}}{d^{N-1}}
=\frac{c/q-ab}{d}\sum_{j,\,l,  \, n \geq 0}a^{j}b^{n-j-l}c^{l}
d^{-n-l}
% \times \\
\frac{q^{(n+1)(n+2)/2+l(l-1)/2}}{(q)_{j}(q)_{n-j-l} (q)_{l}}\\
&=\frac{c/q-ab}{d}\sum_{n\geq 0} (b/d)^{n}q^{(n+1)(n+2)/2}
\sum_{j =0}^{n}\frac{(a/b)^{j}}{(q)_{j}(q)_{n-j}} \\
&\phantom{sdasdasdasdassaadasddassd} \times \sum_{l=0}^{n-j}
q^{l(l-1)/2}(c/bd)^{l}
 \left [
\begin{matrix}
n-j\\
l
\end{matrix}
\right ]\\
&=\frac{c-abq}{dq}\sum_{n\geq 0} (b/d)^{n}q^{(n+1)(n+2)/2}
\sum_{j =0}^{n}\frac{(a/b)^{j}(-c/bd)_{n-j}}{(q)_{j}(q)_{n-j}} \\
&=\frac{c-abq}{dq}\sum_{n\geq 0} (b/d)^{n}q^{(n+1)(n+2)/2}
\sum_{j =0}^{n}\frac{(a/b)^{n-j}(-c/bd)_{j}}{(q)_{j}(q)_{n-j}}  \\
&=\frac{c-abq}{dq}\sum_{j
=0}^{\infty}\frac{(b/a)^{j}(-c/bd)_{j}}{(q)_{j}} \sum_{n\geq j}
\frac{(a/d)^{n}q^{(n+1)(n+2)/2}}{(q)_{n-j}}\\
&=\frac{c-abq}{dq}\sum_{j
=0}^{\infty}\frac{(b/a)^{j}(-c/bd)_{j}}{(q)_{j}} \sum_{n\geq 0}
\frac{(a/d)^{n+j}q^{n(n+3)/2+jn+(j+1)(j+2)/2}}{(q)_{n}}\\
&=\frac{c-abq}{dq}\sum_{j
=0}^{\infty}\frac{(b/d)^{j}(-c/bd)_{j}q^{(j+1)(j+2)/2}}{(q)_{j}}
\sum_{n\geq 0}
\frac{(aq^{j+2}/d)^{n}q^{n(n-1)/2}}{(q)_{n}}\\
&=\frac{c-abq}{dq}\sum_{j
=0}^{\infty}\frac{(b/d)^{j}(-c/bd)_{j}q^{(j+1)(j+2)/2}}{(q)_{j}}
(-aq^{j+2}/d)_{\infty}\\
&=\frac{c-abq}{dq}(-aq^2/d)_{\infty} \sum_{j
=0}^{\infty}\frac{(b/d)^{j}(-c/bd)_{j}\,q^{(j+1)(j+2)/2}}{(q)_{j}(-aq^2/d)_{j}}
.
\end{align*}
} It follows similarly from \eqref{CNeq} that {\allowdisplaybreaks
\begin{align*}
\lim_{N \to \infty} \frac{C_{N}}{d^{N-1}}& =\sum_{j,\, l,\,n\geq
0}a^{j}b^{n-j-l}c^{l}
d^{-n-l}\frac{q^{n(n+1)/2+l(l-1)/2} }{(q)_{j}(q)_{n-j-l}(q)_{l}}\\
&=(-aq/d)_{\infty} \sum_{j
=0}^{\infty}\frac{(b/d)^{j}(-c/bd)_{j}\,q^{j(j+1)/2}}{(q)_{j}(-aq/d)_{j}}
.
\end{align*}
} We omit the details. That \eqref{H1lim} holds is now immediate.

That \eqref{CNlim} and \eqref{DNlim} hold follows immediately upon
letting $d=1$ in the limits above.
\end{proof}

\section{New Proofs of Some Continued Fraction Identities}

In what follows, we make some use of the Jacobi triple product
identity.
\begin{theorem}
For $|q|<1$ and $z \in \mathbb{C}\backslash \{0\}$,
{\allowdisplaybreaks
\[
(-qz;q^{2})_{\infty}(-q/z;q^{2})_{\infty}(q^{2};q^{2})_{\infty}=
\sum_{n=-\infty}^{\infty}z^{n}q^{n^{2}}.
\]
}
\end{theorem}
We also need a result of Watson on basic hypergeometric series. An
$_{r} \phi _{s}$ basic hypergeometric series is defined by
{\allowdisplaybreaks
\begin{multline*} _{r} \phi _{s} \left (
\begin{matrix}
a_{1}, a_{2}, \dots, a_{r}\\
b_{1}, \dots, b_{s}
\end{matrix}
; q,x \right ) = \\
\sum_{n=0}^{\infty} \frac{(a_{1};q)_{n}(a_{2};q)_{n}\dots
(a_{r};q)_{n}} {(q;q)_{n}(b_{1};q)_{n}\dots (b_{s};q)_{n}} \left(
(-1)^{n} q^{n(n-1)/2} \right )^{s+1-r}x^{n},
\end{multline*}
}for $|q|<1$.

Watson's theorem is that
\begin{multline*}
_{8} \phi _{7} \left (
\begin{matrix}
A,\,q\sqrt{A},\,-q\sqrt{A},\,B,\,C,\,D,\,E,\,q^{-n}\\
\sqrt{A},\,-\sqrt{A},\,Aq/B,\,Aq/C,\,Aq/D,\,Aq/E,\,Aq^{n+1}\,
\end{matrix}
; q,\frac{A^{2}q^{n+2}}{BCDE} \right ) = \\
%\phantom{a}\\
 \frac{(A q)_{n} (A q/DE)_{n}  }{(A q/D)_{n} (A q/E)_{n}
}   \,\,   _{4}\phi _{3} \left (
\begin{matrix}
Aq/BC,D,E,q^{-n}\\
Aq/B,Aq/C,DEq^{-n}/A
\end{matrix}\,
; q,q \right ),
\end{multline*}
where $n$ is a non-negative integer. If we let $B$, $D$ and $n \to
\infty$ (as in \cite{H74}), we get
\begin{multline*}
\sum_{r \geq 0} \frac{(1-Aq^{2r})(A)_{r}(C)_{r}(E)_{r}\left(
-A^{2}/CE \right )^{r} q^{3r(r-1)/2+2r}}
{(1-A)(Aq/C)_{r}(Aq/E)_{r}(q)_{r}}\\
= \frac{(Aq)_{\infty}}{(Aq/E)_{\infty}} \sum_{r \geq 0}\frac{
(E)_{r} (-Aq/E)^{r}q^{r(r-1)/2}}{(q)_{r}(Aq/C)_{r}}.
\end{multline*}
If we set $A=c/d^{2}$, $C=-c/ad$ and $E=-c/bd$, we get that
{\allowdisplaybreaks
\begin{align}\label{Wat1} &\lim_{N \to \infty}
\frac{C_{N}}{d^{N-1}}=(-aq/d)_{\infty} \sum_{j
=0}^{\infty}\frac{(bq/d)^{j}(-c/bd)_{j}\,q^{j(j-1)/2}}{(q)_{j}(-aq/d)_{j}}\\
& =\frac{(-aq/d)_{\infty}(-bq/d)_{\infty}}{(cq/d^{2})_{\infty}}\notag\\
 &\times
\sum_{r \geq 0}
\frac{(1-cq^{2r}/d^{2})(-c/ad)_{r}(-c/bd)_{r}(c/d^{2})_{r}\left(
-ab/d^{2} \right )^{r} q^{3r(r-1)/2+2r}}
{(1-c/d^{2})(-aq/d)_{r}(-bq/d)_{r}(q)_{r}}. \notag
\end{align}
}Likewise, if we set $A=cq/d^{2}$, $C=-c/ad$ and $E=-c/bd$, we get
that {\allowdisplaybreaks
\begin{align}\label{Wat2}
&\lim_{N \to \infty} \frac{D_{N}-C_{N}}{d^{N-1}}
=\frac{c-abq}{d}(-aq^2/d)_{\infty} \sum_{j
=0}^{\infty}\frac{(bq^{2}/d)^{j}(-c/bd)_{j}\,q^{j(j-1)/2}}{(q)_{j}(-aq^2/d)_{j}}\\
&=\frac{c-abq}{d}\frac{(-aq^{2}/d)_{\infty}(-bq^{2}/d)_{\infty}}{(cq^{2}/d^{2})_{\infty}}\notag\\
 &\times
\sum_{r \geq 0}
\frac{(1-cq^{2r+1}/d^{2})(-c/ad)_{r}(-c/bd)_{r}(cq/d^{2})_{r}\left(
-abq^{2}/d^{2} \right )^{r} q^{3r(r-1)/2+2r}}
{(1-cq/d^{2})(-aq^{2}/d)_{r}(-bq^{2}/d)_{r}(q)_{r}}. \notag
\end{align}
} We note the first equalities in \eqref{Wat1} and \eqref{Wat2}
imply that the somewhat amusing identity
\[
 \frac{1}{1} \+ \frac{abq+bd}{(a-b)q+d}
 \+ \cds \+ \frac{ab
q^{2n+1}+bdq^n}{(a-b) q^{n+1}+d}\+ \cds = \frac{1}{1+b}
\]
holds for all complex numbers $a$, $b$ and $d$ and all $q$ with
$|q|<1$ and $d \not = -aq^{n}$, $n \geq 1$. This follows upon
setting $c=-bd$ and then replacing $b$ by $-b$. This result also
follows from the following theorem of Pincherle \cite{P94} :

%               Theorem 1

\begin{theorem}\label{T:P*}
(Pincherle)  Let $ \{\,a_{n}\}_{n = 1}^{\infty}$, $\{\,b_{n}\}_{n
= 1}^{\infty}$ and
 $ \{\,G_{n}\}_{n = -1}^{\infty}$ be \\ sequences of real or complex
  numbers satisfying  $a_{n} \neq 0$ for $n\geq 1$
 and for all \,\, $n \geq 1$,
\begin{equation}\label{E:pn}
 G_{n} =   a_{n}G_{n-2}  + b_{n} G_{n - 1}.
\end{equation}
Let $\{B_{n}\}_{n=1}^{\infty}$ denote the denominator convergents
of the continued fraction
 $\K{n = 1}{\infty}\displaystyle{\frac{a_{n}}{b_{n}}}$.

If $\lim_{n \to \infty}\displaystyle{G_{n}/B_{n}} = 0$ then  $\K{n
= 1}{\infty}\displaystyle{\frac{a_{n}}{b_{n}}}$ converges and  its
limit is $-G_{0}/G_{-1}$.
\end{theorem}
However, Pincherle's theorem is less informative in that it does
not show that the numerators converge to $(-aq;q)_{\infty}$ and
that the denominators converge to $(1+b)(-aq;q)_{\infty}$.

The symmetry in $a$ and $b$ in the continued fraction of Theorem
\ref{1/qth} can be exploited to prove something a little less
trivial.
\begin{corollary}\label{cram}
Let $a$, $b$, $c$ and $q$ be complex numbers with $|q|<1$ and
$a\not =0$. Then
\begin{equation}\label{absym1}
(-aq)_{\infty} \sum_{j
=0}^{\infty}\frac{(bq)^{j}(-c/b)_{j}\,q^{j(j-1)/2}}{(q)_{j}(-aq)_{j}}
=(-bq)_{\infty} \sum_{j
=0}^{\infty}\frac{(aq)^{j}(-c/a)_{j}\,q^{j(j-1)/2}}{(q)_{j}(-bq)_{j}}.
\end{equation}
If, in addition,   $1-bq^{n} \not =0$ for $n \geq 1$, then
\begin{equation}\label{rameq}
 \sum_{j
=0}^{\infty}\frac{(-b/a;q)_{j}a^{j}\,q^{j(j+1)/2}}{(q)_{j}(bq)_{j}}
=\frac{(-aq;q)_{\infty}}{(bq;q)_{\infty}}.
\end{equation}
\end{corollary}
The identity at \eqref{absym1} is found in Ramanujan's lost
notebook \cite{S88} and a proof can be found in the recent book by
Andrews and Berndt \cite{AB05}.

The second  identity is found in Ramanujan's notebooks (see
\cite{B94}, Chapter 27, Entry 1, page 262). This identity is also
equivalent to a result found in Andrews \cite{A72}, where Andrews
attributes it to Cauchy.
\begin{proof}
Let $d=1$ in Theorem \ref{1/qth}. The symmetry in $a$ and $b$ in
the continued fraction in Theorem \ref{1/qth} and the first
equality in \eqref{Wat1} give \eqref{absym1} immediately.

Ramanujan's  result \eqref{rameq} follows  after setting $c=-b$
and then replacing $b$ by $-b$.
\end{proof}

We now prove some continued fraction identities.
\begin{corollary}\label{c1}
If $|q|<1$, then
\begin{equation}\label{q2q3}
\frac{1}{1} \-\frac{q}{q+1}\-\frac{q^3}{q^{2}+1}\-\cds
\-\frac{q^{2n-1}}{q^{n}+1}\-\cds
=\frac{(q^{2};q^{3})_{\infty}}{(q;q^{3})_{\infty}},
\end{equation}
with the numerators converging to $1/(q;q^{3})_{\infty}$ and the
denominators converging to $1/(q^{2};q^{3})_{\infty}$.
\end{corollary}
This continued fraction is due to Ramanujan and can be found in
his second notebook (\cite{R57}, page 290). It has been proved in
\cite{ABJL92}  and in \cite{ABSYZ03}. Both proofs are quite
difficult and it is not clearly evident that the numerators and
denominators converge separately. We feel this present proof is
simpler, although it uses the power of Watson's Theorem.
\begin{proof}
Let $\omega = \exp(2 \pi \imath/3)$ and set $a=-\omega$, $b= -
\omega^{2}$, $c=0$ and $d=1$ in Theorem \ref{1/qth}, so that the
continued fraction in the theorem is the continued fraction in
Corollary \ref{c1}. By the second equality in \eqref{Wat1},
 {\allowdisplaybreaks
\begin{align*}
\lim_{n \to \infty} C_{n} &= (\omega q)_{\infty}(\omega^{2}
q)_{\infty} \sum_{r \geq 0} \frac{\left( -q^{2} \right )^{r}
q^{3r(r-1)/2}} {(\omega q)_{r}(\omega^{2}q)_{r}(q)_{r}}\\
& = (\omega q)_{\infty}(\omega^{2} q)_{\infty} \sum_{r \geq 0}
\frac{\left( -q^{2} \right )^{r} (q^{3})^{r(r-1)/2}}
{(q^{3};q^{3})_{r}}\\ & = (\omega q)_{\infty}(\omega^{2}
q)_{\infty}(q^{2};q^{3})_{\infty}\\
&= \frac{(\omega q)_{\infty}(\omega^{2}
q)_{\infty}( q)_{\infty}(q^{2};q^{3})_{\infty}}{( q)_{\infty}}\\
&= \frac{(q^{3};q^{3})_{\infty}(q^{2};q^{3})_{\infty}}{(
q)_{\infty}}\\
&=\frac{1}{(q;q^{3})_{\infty}}.
\end{align*}
} The third equality above follows from the $q$-binomial theorem.
By the second equality in \eqref{Wat2} {\allowdisplaybreaks
\begin{align*} \lim_{n \to \infty}
D_{n}-C_{n}& =-q(\omega q^{2})_{\infty}(\omega^{2} q^{2})_{\infty}
\sum_{r \geq 0} \frac{\left( -1 \right )^{r} q^{3r(r-1)/2+4r}}
{(\omega q^{2})_{r}(
\omega^{2} q^{2})_{r}(q)_{r}}\\
& =-q(\omega q)_{\infty}(\omega^{2} q)_{\infty} \sum_{r \geq 0}
\frac{\left( -1 \right )^{r} q^{3r(r-1)/2+4r}(1-q^{r+1})} {(\omega
q)_{r+1}(
\omega^{2} q)_{r+1}(q)_{r+1}}\\
& =q(\omega q)_{\infty}(\omega^{2} q)_{\infty} \sum_{r \geq 1}
\frac{\left( -1 \right )^{r} q^{(3r^2-r)/2-1}(1-q^{r})} {(\omega
q)_{r}( \omega^{2} q)_{r}(q)_{r}}\\
& =(\omega q)_{\infty}(\omega^{2} q)_{\infty} \sum_{r \geq 1}
\frac{\left( -1 \right )^{r} q^{(3r^2-3r)/2}(q^{r}-q^{2r})}
{(\omega
q)_{r}( \omega^{2} q)_{r}(q)_{r}}\\
& =(\omega q)_{\infty}(\omega^{2} q)_{\infty} \sum_{r \geq 0}
\frac{\left( -1 \right )^{r} (q^{3})^{r(r-1)/2}(q^{r}-q^{2r})}
{(q^{3};q^{3})_{r}}\\
& =(\omega q)_{\infty}(\omega^{2}
q)_{\infty}((q;q^{3})_{\infty}-(q^{2};q^{3})_{\infty}).
\end{align*}
} Here again we have used the $q$-binomial theorem. From the third
expression above for $\lim_{n \to \infty}C_{n}$, it follows that
{\allowdisplaybreaks
\begin{align*} \lim_{n \to \infty} D_{n}&
=(\omega q)_{\infty}(\omega^{2}
q)_{\infty}(q;q^{3})_{\infty}\\
& =\frac{(\omega q)_{\infty}(\omega^{2} q)_{\infty}(
q)_{\infty}(q;q^{3})_{\infty}}{( q)_{\infty}}\\
& =\frac{(q^{3};q^{3})_{\infty}(q;q^{3})_{\infty}}{( q)_{\infty}}\\
& =\frac{1}{(q^{2}; q^{3})_{\infty}}.
\end{align*}
}
\end{proof}

\begin{corollary}\label{c2}
If $|q|<1$, then
%{\allowdisplaybreaks
\begin{align}\label{z3}
S(q):= \frac{1}{1} \+
 \frac{q+q^{2}}{1}
\+
 \frac{q^{2}+q^{4}}{1}
\+
 \frac{q^{3}+q^{6}}{1}
%\+
 %\frac{q^{4}+q^{8}}{1}
\+ \cds =\frac{(q;q^{2})_{\infty}}{(q^{3};q^{6})_{\infty}^{3}}.
\end{align}
%}
\end{corollary}
This continued fraction is also due to Ramanujan and can be found
in his second notebook (\cite{R57}, page 373). The first proofs in
print are due to Watson \cite{W29} and Selberg \cite{S36}. Other
proofs are due to Gordon \cite{G65}, Andrews \cite{A68} and
Hirschhorn \cite{H80}.
\begin{proof}
Set $a=-1/q^{1/2}$, $b= 1/q^{1/2}$, $c=1$ and $d=1$ in Theorem
\ref{1/qth}, so that the continued fraction in the theorem is \[
\frac{1}{1+2S(q)}.\] By the second equality in \eqref{Wat1},
{\allowdisplaybreaks
\begin{align*}
\lim_{n \to \infty} C_{n} &= \frac{( q^{1/2})_{\infty}(-q^{1/2}
)_{\infty} }{(q)_{\infty}} \\
&\times\left ( 1+\sum_{r \geq 1}
\frac{(1-q^{2r})(q^{1/2})_{r}(-q^{1/2})_{r}(q)_{r-1}
q^{3r(r-1)/2+r}} {( q^{1/2})_{r}(-q^{1/2})_{r}(q)_{r}}\right)\\
&= \frac{( q;q^{2})_{\infty} }{(q)_{\infty}} \left ( 1+\sum_{r \geq 1}(1+q^{r}) q^{(3r^{2}-r)/2}\right)\\
&= \frac{( q;q^{2})_{\infty} }{(q)_{\infty}}\sum_{r =-\infty}^{\infty} (q^{3/2})^{r^{2}}(q^{1/2})^{r}\\
&= \frac{( q;q^{2})_{\infty}
}{(q)_{\infty}}(-q;q^{3})_{\infty}(-q^{2};q^{3})_{\infty}(q^{3};q^{3})_{\infty}.
\end{align*}
} The last equality above follows from the Jacobi triple product
identity. By the second equality in \eqref{Wat2},
{\allowdisplaybreaks
\begin{align*} \lim_{n \to \infty}
D_{n}-C_{n}&= 2\frac{( q^{3/2})_{\infty}(-q^{3/2} )_{\infty}
}{(q^{2})_{\infty}}
\\ &\times
\sum_{r \geq 0}
\frac{(1-q^{2r+1})(q^{1/2})_{r}(-q^{1/2})_{r}(q)_{r}
q^{3r(r-1)/2+3r}} {(1-q)( q^{3/2})_{r}(-q^{3/2})_{r}(q)_{r}}\\
&= 2\frac{( q;q^{2})_{\infty} }{(q)_{\infty}}
\sum_{r \geq 0} q^{3r^{2}/2+3r/2}\\
&= \frac{( q;q^{2})_{\infty} }{(q)_{\infty}} \sum_{r
=-\infty}^{\infty}
(q^{3/2})^{r^{2}}(q^{3/2})^{r}\\
&= \frac{( q;q^{2})_{\infty} }{(q)_{\infty}} (-1;q^{3})_{\infty}(-q^{3};q^{3})_{\infty}(q^{3};q^{3})_{\infty}\\
&= 2\frac{( q;q^{2})_{\infty} }{(q)_{\infty}}
(-q^{3};q^{3})_{\infty}(-q^{3};q^{3})_{\infty}(q^{3};q^{3})_{\infty}.
\end{align*} }
Here again we have used the Jacobi triple product theorem.

The result now follows after a little algebra and some elementary
infinite product manipulations.
\end{proof}
We next give a proof of the Rogers-Ramanujan identities.
Hirschhorn has given a proof in \cite{HT80} of these identities
that is equivalent to setting $a=b=0$, $c=1$ and $d=0$ in
  \eqref{Wat1} and \eqref{Wat2}, whereupon the identities follow after applying the Jacobi
Triple product identity. We include this proof because of its
elegance.
 We also give a different proof for the infinite product representations
for the numerator and denominator (see Corollary \ref{c3} below).
For this we need two identities, one due to Rogers \cite{R17} and
one due to Andrews \cite{A85} (identities A.44 and A.62 in
Slater's list \cite{S52}): {\allowdisplaybreaks
\begin{equation}\label{a44}
\sum_{r=0}^{\infty}\frac{q^{3r(r+1)/2}}{(q;q^{2})_{r+1}(q;q)_{r}}=
\frac{(q^{8};q^{10})_{\infty}(q^{2};q^{10})_{\infty}(q^{10};q^{10})_{\infty}}{(q;q)_{\infty}},
\end{equation}
} {\allowdisplaybreaks
\begin{equation}\label{a46}
\sum_{r=0}^{\infty}\frac{(-q;q)_{r}q^{r(3r+1)/2}}{(q;q)_{2r+1}}=
\frac{(q^{6};q^{10})_{\infty}(q^{4};q^{10})_{\infty}(q^{10};q^{10})_{\infty}}{(q;q)_{\infty}}.
\end{equation}
}

\begin{corollary}\label{c3}
Let $A_{n}(q)$ and $B_{n}(q)$ denote the $n$-th numerator and
denominator, respectively of the continued fraction
\[
K(q)=1+
 \frac{q}{1}
\+
 \frac{q^{2}}{1}
\+
 \frac{q^{3}}{1}
\+\,\cds.
\]
Then {\allowdisplaybreaks
\begin{align}\label{rrids}
\lim_{n \to \infty}
A_{n}(q)&=\sum_{n=0}^{\infty}\frac{q^{n^{2}}}{(q;q)_{n}}=\frac{1}{(q;q^{5})_{\infty}(q^{4}
;q^{5})_{\infty}}
,\\
&\phantom{as} \notag \\
\lim_{n \to \infty}
B_{n}(q)&=\sum_{n=0}^{\infty}\frac{q^{n^{2}+n}}{(q;q)_{n}}
=\frac{1}{(q^{2};q^{5})_{\infty}(q^{3};q^{5})_{\infty}}. \notag
\end{align}
}
\end{corollary}
These identities were first proved by L.J. Rogers in 1894
\cite{R94} in a paper that was completely ignored. They were
rediscovered (without proof) by Ramanujan sometime before 1913. In
1917, Ramanujan rediscovered Rogers's paper while browsing a
journal. Also in 1917, these identities were rediscovered and
proved independently by Issai Schur \cite{S17}. There are now many
different proofs.
\begin{proof}
First set $a=b= 0$, $c=1$ and $d=1$ in Theorem \ref{1/qth}, so
that the continued fraction in the theorem is \[ \frac{1}{1}
\+\frac{1}{K(q)}. \] By the first equalities at \eqref{Wat1} and
\eqref{Wat2}, {\allowdisplaybreaks
\begin{align*}
\lim_{n \to \infty} A_{n}(q)&=\lim_{n \to \infty} C_{n}=
\sum_{r\geq 0} \frac{  q^{r} q^{r(r-1)/2} q^{r(r-1)/2 }}
{(q)_{r}}= \sum_{r\geq
0} \frac{   q^{r^{2} }} {(q)_{r}},\\
\lim_{n \to \infty} B_{n}(q)&=\lim_{n \to \infty}D_{n}- C_{n}=
\sum_{r\geq 0} \frac{  q^{2r} q^{r(r-1)/2} q^{r(r-1)/2 }}
{(q)_{r}}= \sum_{r\geq 0} \frac{   q^{r^{2}+r }} {(q)_{r}}.
\end{align*}
} This proves the first equalities in each case of \eqref{rrids}.

By the second equalities at \eqref{Wat1} and \eqref{Wat2},
{\allowdisplaybreaks
\begin{align*}
\lim_{n \to \infty} A_{n}(q)&=\lim_{n \to \infty} C_{n}\\
&=\frac{1}{(q)_{\infty}} \left(1+ \sum_{r\geq 1} \frac{ (1-
q^{2r}) (-1)^{r} (q)_{r-1}q^{5r^{2}/2-r/2}} {(q)_{r}} \right)
\\
&= \frac{1}{(q)_{\infty}}\sum_{r=-\infty}^{\infty}
(-q^{1/2})^{r} (q^{5/2})^{r^{2}}\\
&= \frac{(q^{2};q^{5})_{\infty}(q^{3};q^{5})_{\infty}(q^{5};q^{5})_{\infty}}{(q)_{\infty}}\\
&= \frac{1}{(q;q^{5})_{\infty}(q^{4};q^{5})_{\infty}}.\\
\lim_{n \to \infty} B_{n}(q)&=\lim_{n \to \infty}D_{n}- C_{n}\\
&=\frac{1}{(q^{2})_{\infty}}\sum_{r\geq
0} \frac{ (1- q^{2r+1}) (-1)^{r} (q)_{r}q^{5r^{2}/2+3r/2}} {(1-q)(q)_{r}} \\
&= \frac{1}{(q)_{\infty}}\sum_{r=-\infty}^{\infty}
(-q^{3/2})^{r} (q^{5/2})^{r^{2}}\\
&= \frac{(q;q^{5})_{\infty}(q^{4};q^{5})_{\infty}(q^{5};q^{5})_{\infty}}{(q)_{\infty}}\\
&= \frac{1}{(q^{2};q^{5})_{\infty}(q^{3};q^{5})_{\infty}}.
\end{align*}
} This proves the second equalities in each case of \eqref{rrids}.
The penultimate equalities in each case above come from the Jacobi
triple product identity. This is Hirschhorn's proof in
\cite{HT80}.

However we desire to give an alternative proof for the infinite
product identities so let $-a=b=q^{1/2}$, $c=0$ and $d=1$ in
Theorem \ref{1/qth}, so that the continued fraction in the theorem
becomes $1/K(q^{2})$. Thus, from the second equality in
\eqref{Wat1}, {\allowdisplaybreaks \begin{align*} \lim_{n \to
\infty}B_{n}(q^{2}) = \lim_{n \to \infty}C_{n}
&=(-q^{3/2})_{\infty}(q^{3/2})_{\infty}
\sum_{r \geq 0} \frac{q^{3r(r-1)/2+3r}} {( q^{3/2})_{r}(-q^{3/2})_{r}(q)_{r}}\\
&=(q;q^{2})_{\infty}\sum_{r \geq 0} \frac{q^{3r(r+1)/2}} {(q;q^{2})_{r+1}(q)_{r}}\\
&=(q;q^{2})_{\infty}\frac{(q^{8};q^{10})_{\infty}(q^{2};q^{10})_{\infty}(q^{10};q^{10})_{\infty}}
{(q;q)_{\infty}}\\
&=\frac{1}{(q^{4};q^{10})_{\infty}(q^{6};q^{10})_{\infty}}.
\end{align*}
} For the third equality we have used \eqref{a44}. From the second
equality in \eqref{Wat1}, {\allowdisplaybreaks \begin{align*}
 \lim_{n \to \infty}D_{n}-C_{n}
&=q^{2}(-q^{5/2})_{\infty}(q^{5/2})_{\infty}
\sum_{r \geq 0} \frac{q^{3r(r-1)/2+5r}} {( q^{5/2})_{r}(-q^{5/2})_{r}(q)_{r}}\\
&=q^{2}(q;q^{2})_{\infty}\sum_{r \geq 0} \frac{q^{3r(r-1)/2+5r}} {(q;q^{2})_{r+2}(q)_{r}}\\
&=q^{2}(q;q^{2})_{\infty}\sum_{r \geq 0} \frac{q^{3r(r-1)/2+5r}(1-q^{r+1})} {(q;q^{2})_{r+2}(q)_{r+1}}\\
&=(q;q^{2})_{\infty}\sum_{r \geq 1} \frac{q^{(3r^{2}+r)/2}(1-q^{r})} {(q;q^{2})_{r+1}(q)_{r}}\\
&=(q;q^{2})_{\infty}\sum_{r \geq 0} \frac{q^{(3r^{2}+r)/2}(1-q^{r})} {(q;q^{2})_{r+1}(q)_{r}}\\
\Rightarrow \lim_{n \to \infty}A_{n}(q^{2})=\lim_{n \to
\infty}D_{n}
&=(q;q^{2})_{\infty}\sum_{r \geq 0} \frac{q^{(3r^{2}+r)/2}} {(q;q^{2})_{r+1}(q)_{r}}\\
&=(q;q^{2})_{\infty}\sum_{r \geq 0} \frac{(-q;q)_{r}q^{(3r^{2}+r)/2}} {(q)_{2r+1}}\\
&=(q;q^{2})_{\infty}\frac{(q^{6};q^{10})_{\infty}(q^{4};q^{10})_{\infty}(q^{10};q^{10})_{\infty}}
{(q;q)_{\infty}}\\
&=\frac{1}{(q^{2};q^{10})_{\infty}(q^{8};q^{10})_{\infty}}.
\end{align*}
} The result now follows after replacing $q^{2}$ by $q$.
\end{proof}

%\section{Some new continued fraction identities}
%We next use Slater's list of identities in conjunction
%with Theorems \ref{qth} and \ref{1/qth}  to find some new continued fraction identities.

\section{A generalization of a continued fraction of Ramanujan}

Before coming to our results in this section we need to state some
notation and to recall some other necessary results.

We call $d_{0}+K_{n=1}^{\infty}c_{n}/d_{n}$ a \emph{canonical
contraction} of
 $b_{0}+K_{n=1}^{\infty}a_{n}/b_{n}$ if
\begin{align*}
&C_{k}=A_{n_{k}},& &D_{k}=B_{n_{k}}& &\text{ for }
k=0,1,2,3,\ldots \, ,\phantom{asdasd}&
\end{align*}
where $C_{n}$, $D_{n}$, $A_{n}$ and $B_{n}$ are canonical
numerators and denominators of $d_{0}+K_{n=1}^{\infty}c_{n}/d_{n}$
and $b_{0}+K_{n=1}^{\infty}a_{n}/b_{n}$ respectively. From
\cite{LW92} (page 85) we  have:
\begin{theorem}\label{odcf}
The canonical contraction of $b_{0}+K_{n=1}^{\infty}a_{n}/b_{n}$
with $C_{0}=A_{1}/B_{1}$
\begin{align*}
&C_{k}=A_{2k+1}& &D_{k}=B_{2k+1}& &\text{ for } k=1,2,3,\ldots \,
,&
\end{align*}
exists if and only if $b_{2k+1} \not = 0 for K=0,1,2,3,\ldots$,
and in this case is given by
\begin{multline}\label{E:odcf}
\frac{b_{0}b_{1}+a_{1}}{b_{1}} -
\frac{a_{1}a_{2}b_{3}/b_{1}}{b_{1}(a_{3}+b_{2}b_{3})+a_{2}b_{3}}
\-
\frac{a_{3}a_{4}b_{5}b_{1}/b_{3}}{a_{5}+b_{4}b_{5}+a_{4}b_{5}/b_{3}}\\
\- \frac{a_{5}a_{6}b_{7}/b_{5}}{a_{7}+b_{6}b_{7}+a_{6}b_{7}/b_{5}}
\- \frac{a_{7}a_{8}b_{9}/b_{7}}{a_{9}+b_{8}b_{9}+a_{8}b_{9}/b_{7}}
\+ \cds .
\end{multline}
\end{theorem}
The continued fraction \eqref{E:odcf} is called the \emph{odd}
part of $b_{0}+K_{n=1}^{\infty}a_{n}/b_{n}$.

We will also make use of the following theorem of Worpitzky (see
\cite{LW92}, pp. 35--36).

\begin{theorem}(Worpitzky)
 Let the continued fraction $K_{n=1}^{\infty}a_{n}/1$ be such that
$|a_{n}|\leq 1/4$ for $n \geq 1$. Then $K_{n=1}^{\infty}a_{n}/1$
converges.
 All approximants of the continued fraction lie in the disc $|w|<1/2$ and the value of the
continued fraction is in the disk $|w|\leq1/2$.\end{theorem}

The following identity can be found in Ramanujan's notebooks
\cite{R57} (a proof can be found in \cite{B98}).

\textbf{Entry17 (p.374)}. Let $a$, $b$ and $q$ be complex numbers,
with $|q|<1$. Define
\begin{equation*}
\phi(a) =
\sum_{n=0}^{\infty}\frac{q^{(n^2+n)/2}a^{n}}{(q;q)_{n}(-bq;q)_{n}}.
\end{equation*}
Then
\begin{equation}\label{phicf}
\frac{\phi(a)}{\phi(aq)} = 1+ \frac{aq}{1} \+ \frac{bq}{1} \+
\frac{aq^2}{1} \+ \frac{bq^2}{1} \+ \frac{aq^3}{1} \+
\frac{bq^3}{1} \+ \cds .
\end{equation}
The continued fraction in \eqref{phicf} is the special case
$\lambda =0$ of the following more general continued fraction,
which can be found in the lost notebook \cite{S88}. Set
\[
F(a,b,\lambda,q) = 1+ \frac{aq+\lambda q}{1} \+ \frac{bq+\lambda
q^2}{1} \+ \frac{aq^2+\lambda q^3}{1} \+ \frac{bq^2+\lambda
q^4}{1} \+ \cds
\]
and
\[
G(a,b,\lambda, q)= \sum_{n=0}^{\infty} \frac{(-\lambda
/a)_{n}q^{(n^2+n)/2}a^{n}} {(q;q)_{n}(-bq;q)_{n}}.
\]
Then
\begin{equation}\label{FGeq}
F(a,b,\lambda, q)= \frac{G(a,b,\lambda,q)}{G(aq,b,\lambda q,q)}.
\end{equation}
A number of authors have given proofs of \eqref{FGeq}, including
Hirschhorn \cite{H80}.

We now evaluate two other continued fractions which also
specialize to give the continued fraction in \eqref{phicf}.

\begin{corollary}\label{cphi}
Let $a$, $b$, $e$ and $q$ be complex numbers with $|q|<1$ and
$|e|<1/4$.

(i) Set \[ \phi(x)=\sum_{j
=0}^{\infty}\frac{\left(\frac{x}{e+1}\right)^{j}\left(\frac{e
aq}{x(e+1)}\right)_{j}\,q^{j(j+1)/2}}{(q)_{j}\left (\frac{-b
q}{e+1} \right )_{j}}, \] and define
\[H_{2}(a,b,e,q):=1+ \frac{aq}{1} \+ \frac{bq+e}{1} \+ \frac{aq^2}{1}
\+ \frac{bq^2+e}{1} \+ \frac{aq^3}{1} \+ \frac{bq^3+e}{1} \+ \cds
.
\]
If  $e \not = - b q^{n}$, for $n \geq 1$, then
\begin{equation}\label{phicfn}
H_{2}(a,b,e,q)=\frac{\phi(a)}{\phi(aq)}.
\end{equation}

(ii) Set \[ \phi(x)=\sum_{j
=0}^{\infty}\frac{\left(\frac{x}{e+1}\right)^{j}\left(\frac{e
b}{a(e+1)}\right)_{j}\,q^{j(j+1)/2}}{(q)_{j}\left (\frac{-b
q}{e+1} \right )_{j}}, \] and define
\[H_{3}(a,b,e,q):=1+ \frac{aq+e}{1} \+ \frac{bq}{1} \+ \frac{aq^2+e}{1}
\+ \frac{bq^2}{1} \+ \frac{aq^3+e}{1} \+ \frac{bq^3}{1} \+ \cds .
\]
If  $e \not = - a q^{n}$, for $n \geq 1$, then
\begin{equation}\label{phicfn2}
H_{3}(a,b,e,q)=(e+1)\frac{\phi(a)}{\phi(aq)}.
\end{equation}

\end{corollary}
Remarks: (i) Note that the value $e=0$ gives Ramanujan's identity
\eqref{phicf} in each case.\\
(ii) The results may also hold for some $e$ satisfying $|e| \geq
1/4$, but we do not explore that here.
\begin{proof}
(i) The condition on $e$ ensures, by Worpitzky's Theorem,
 that the continued fraction above converges and hence equals
its odd part, which is {\allowdisplaybreaks
\begin{multline*}
1+aq -\frac{aq(bq+e)}{aq^2+bq+e+1} \-
\frac{aq^2(bq^2+e)}{aq^3+bq^2+e+1} \-
\cds\\
\- \frac{aq^{n+1}(bq^{n+1}+e)}{aq^{n+2}+bq^{n+1}+e+1} \-
\cds\\
=1+aq +\frac{(-aqb)q+(-aqe)}{(aq+b)q+e+1} \+
\frac{(-aqb)q^3+(-aqe)q}{(aq+b)q^2+e+1} \+
\cds\\
\+ \frac{(-aqb)q^{2n+1}+(-aqe)q^{n}}{(aq+b)q^{n+1}+e+1} \+
\cds .\\
\end{multline*}
} Theorem \ref{1/qth} applied to the continued fraction
$H_{1}(b,aq, -aqe,e+1,q)$ gives that {\allowdisplaybreaks
\begin{multline*} \frac{(-aqb)q+(-aqe)}{(aq+b)q+e+1} \+
\frac{(-aqb)q^3+(-aqe)q}{(aq+b)q^2+e+1} \+
\cds\\
\phantom{sdadssaasddsadasdadasdaaasdadad}\+
\frac{(-aqb)q^{2n+1}+(-aqe)q^{n}}{(aq+b)q^{n+1}+e+1} \+
\cds \\
= \frac{ \displaystyle{ \frac{-aq(e+bq)}{e+1} \sum_{j
=0}^{\infty}\frac{(a/(e+1))^{j}(e/(e+1))_{j}\,q^{j(j+5)/2}}{(q)_{j}(-bq/(e+1))_{j+1}}}
} {\displaystyle{ \sum_{j
=0}^{\infty}\frac{(a/(e+1))^{j}(e/(e+1))_{j}\,q^{j(j+3)/2}}{(q)_{j}(-bq/(e+1))_{j}}}}.
\end{multline*}
} Thus {\allowdisplaybreaks
\begin{multline*} H_{2}(a,b,e,q)\\= 1+aq
+\frac{ \displaystyle{ \frac{-aq(e+bq)}{e+1} \sum_{j
=0}^{\infty}\frac{(aq/(e+1))^{j}(e/(e+1))_{j}\,q^{j(j+3)/2}}{(q)_{j}(-bq/(e+1))_{j+1}}}
} {\displaystyle{ \sum_{j
=0}^{\infty}\frac{(aq/(e+1))^{j}(e/(e+1))_{j}\,q^{j(j+1)/2}}{(q)_{j}(-bq/(e+1))_{j}}}}.
\end{multline*}
} Hence the result will be true if {\allowdisplaybreaks
\begin{multline*} (1+a q)\sum_{j
=0}^{\infty}\frac{(aq/(e+1))^{j}(e/(e+1))_{j}\,q^{j(j+1)/2}}{(q)_{j}(-bq/(e+1))_{j}}\\
- \frac{aq(e+bq)}{e+1} \sum_{j
=0}^{\infty}\frac{(aq/(e+1))^{j}(e/(e+1))_{j}\,q^{j(j+3)/2}}{(q)_{j}(-bq/(e+1))_{j+1}}\\
=\sum_{j =0}^{\infty}\frac{\left(a/(e+1)\right)^{j}\left(e
q/(e+1)\right)_{j}\,q^{j(j+1)/2}}{(q)_{j}\left (-b q/(e+1) \right
)_{j}}.
\end{multline*}
} This follows easily by considering both sides as power series in
$a$ and comparing coefficients. The coefficient of $a^{0}$ is
clearly seen to be 1 on each side. For $j\geq 1$, the coefficient
of $a^{j}$ on the left side is {\allowdisplaybreaks
\begin{multline*}
\frac{(q/(e+1))^{j}(e/(e+1))_{j}\,q^{j(j+1)/2}}{(q)_{j}(-bq/(e+1))_{j}}\\
+q
\frac{(q/(e+1))^{j-1}(e/(e+1))_{j-1}\,q^{j(j-1)/2}}{(q)_{j-1}(-bq/(e+1))_{j-1}}\\
\phantom{asdasdaasadasdadassd}-
\frac{q(e+bq)}{e+1}\frac{(q/(e+1))^{j-1}(e/(e+1))_{j-1}\,q^{(j-1)(j+2)/2}}{(q)_{j-1}(-bq/(e+1))_{j}}\\
=\frac{(1/(e+1))^{j}(e/(e+1))_{j-1}\,q^{j(j+1)/2}}{(q)_{j}(-bq/(e+1))_{j}}\times
 \bigg ( \left(1-\frac{e
   q^{j-1}}{e+1}\right) q^j
   \phantom{asdasdaassd}\\
  \phantom{asdasdasdaassd}  +(e+1) \left(1-q^j\right)
   \left(1+\frac{bq^j}{e+1}\right)-(e+bq) \left(1-q^j\right) q^{j-1} \bigg
   )\\
   =\frac{(1/(e+1))^{j}(eq/(e+1))_{j}\,q^{j(j+1)/2}}{(q)_{j}(-bq/(e+1))_{j}}.\phantom{asdasdaasadaasdasdasddasadassd}
\end{multline*}
} (ii) The proof in this case is initially similar to that of (i).
We first show, using similar reasoning to that in part (i), that
\[ H_{3}(a,b,e,q)= \frac{ \displaystyle{ (e+1) \sum_{j
=0}^{\infty}\frac{(b/(e+1))^{j}(e/(e+1))_{j}\,q^{j(j+1)/2}}{(q)_{j}(-aq/(e+1))_{j}}
} } {\displaystyle{ \frac{1}{1+aq/(e+1)}}\sum_{j
=0}^{\infty}\frac{(b/(e+1))^{j}(e/(e+1))_{j}\,q^{j(j+1)/2}}{(q)_{j}(-aq^{2}/(e+1))_{j}}}.
\]
The identity at \eqref{absym1} is now applied to the numerator and
denominator in this last expression and \eqref{phicfn2} follows.
\end{proof}
\emph{Acknowledgements.} We would like to thank the referee for
his careful reading of the paper and his many suggestions which
improved its layout.

 \allowdisplaybreaks{

}

\end{document}